\documentclass{amsart}
\input{anc/preamble.sty}
\title{Dedekind Sums with Even Denominators}
\author{Michael Kural}
\address{Department of Mathematics, Massachusetts Institute of Technology}
\email{mkural@mit.edu}

\begin{document}

\begin{abstract}
Let $S(a,b)$ denote the normalized Dedekind sum. We study the range of possible values for $S(a,b)=\frac{k}{q}$ with $\gcd(k,q)=1$. Girstmair proved local restrictions on $k$ depending on $q\pmod{12}$ and whether $q$ is a square and conjectured that these are the only restrictions possible. We verify the conjecture in the cases $q$ even, $q$ a square divisible by $3$ or $5$, and $2\le q\le 200$ (the latter by computer), and provide progress towards a general approach.
\end{abstract}
\maketitle

\section{Introduction}

\begin{defn}
For coprime integers $a$ and $b$ with $b>0$,  the \textit{Dedekind sum} $s(a,b)$ is defined as
\[
s(a,b) = \sum_{k=1}^{b} \left(\!\!\left(\frac{k}{b}\right)\!\!\right)\left(\!\!\left(\frac{ak}{b}\right)\!\!\right),
\]
where $ {\displaystyle (\!(\,)\!):\mathbb {R} \rightarrow \mathbb {R} }$ denotes the \textit{sawtooth function}, defined by
\[
(\!(x)\!) = \begin{cases}
\{x\}-\frac{1}{2}&x \not \in \Z\\
0 &x \in \Z.
\end{cases}
\]
We will primarily work with the \textit{normalized Dedekind sum} $S(a,b)$, defined by 
\[
S(a,b)=12s(a,b),
\] 
which will make computation more convenient.

\end{defn}

The definition of the Dedekind sum is motivated by its use in the transformation law of the Dedekind eta function (\cite[p. 52]{Apostol76}). Dedekind sums have been studied in a variety of contexts, including in applications to algebraic geometry, lattice point enumeration, and the study of modular forms (\cite{Urzua07,Rademacher72,Apostol76,Bruggeman94}). The distribution of possible values of Dedekind sums has also been considered extensively (\cite{Vardi93,Bruggeman94,Hickerson97,Myerson98,Girstmair15,GirstmairIntegers2017,GirstmairLargest2017,GirstmairEqual2016}). 

It was noted by Rademacher and Grosswald in \cite[p. 28]{Rademacher72} that the range of values of $S(a,b)$ is unknown, which is our central question. Hickerson (\cite{Hickerson97}) proved that this range is dense in $\R$, and Girstmair (\cite{Girstmair15}) found that each rational number $r \in [0,1)$ occurs as the fractional part of a Dedekind sum $S(a,b)$. Furthermore, Girstmair (\cite{GirstmairNovember2017}) proved that each value in this range occurs as a Dedekind sum infinitely many times in a nontrivial sense. Recently, Girstmair (\cite{Girstmair17}) classified the denominator of a Dedekind sum in terms of $a$ and $b$: if $S(a,b) = \frac{k}{q}$ with $\gcd(k,q)=1$, then $q = \frac{b}{\gcd(a^2+1,b)}$. Girstmair also conjectured that for a fixed integer $q\ge 2$, the set of possible $k$ coprime with $q$ such that $\frac{k}{q} = S(a,b)$ for some $a,b$ are exactly the integers $k$ for which 
\begin{itemize}
\item If $3\nmid q$, then $3\mid k$.
\item If $2\nmid q$, then
\[
k \equiv \begin{cases}
2\pmod{4} & q\equiv 3 \pmod{4}\\
0 \pmod{8} & q\text{ is a square}\\
0\pmod{4}&\text{else.}
\end{cases}
\]
\end{itemize}
and showed that these conditions are indeed necessary. Finally, Girstmair proved that if $k\equiv k' \pmod{q(q^2-1)}$, then $\frac{k}{q}$ is the value of a normalized Dedekind sum if and only if $\frac{k'}{q}$ is the value of a normalized Dedekind sum, effectively reducing the problem to a finite existence problem $\pmod{q(q^2-1)}$ for each $q$. In particular, this reduction allowed Girstmair to verify the conjecture for all $q\le 60$ by computer.

In Section \ref{sec:reduce} of the present paper, we establish a decomposition $qS(a',b') = \lambda(a,t,t^{*}) + \Delta(at^{*}+j)$ for certain numerators of normalized Dedekind sums and analyze $\lambda$ and $\Delta$ $\pmod{q}$ and $\pmod{q^2-1}$ to reduce Girstmair's conjecture to an existence problem for $\lambda \pmod{q^2-1}$. In Section \ref{sec:partial}, we specialize $\lambda\pmod{q^2-1}$ to a function $f(a)$, which we then split into a linear and a periodic part based on a generalization of Rademacher's three-term relation used by Girstmair. If the slope of the linear part is small enough (in particular if $\gcd(a,q)$ is small enough), this observation is enough to prove Girstmair's conjecture. In particular, we prove the conjecture holds for all $q$ even and all perfect squares $q$ which are divisible by $3$ or $5$. We also use computer verification and our first observation to prove the conjecture for all $2\le q\le 200$.

As noted in \cite{Girstmair17}, the case $q=1$ has been resolved completely, so for the remainder of the paper we assume $q\ge 2$.

\section{Background}

Note that given a fixed $b$, the normalized Dedekind sum only depends on the residue class of $a\pmod{b}$.

A classical fact about Dedekind sums reveals unexpected symmetry:

\begin{lemma}[Reciprocity Law]
\label{recip}
If $a$ and $b$ are coprime positive integers, then
\[
S(a,b) +S(b,a) = \frac{a}{b}+\frac{b}{a}+\frac{1}{ab}-3.
\]
\end{lemma}
\begin{proof}
See, for example, \cite[p. 27]{Rademacher72}.
\end{proof}
The reciprocity law is crucial to the study of Dedekind sums. As an example, along with the fact that $S(a-nb,b) = S(a,b)$ for $n \in \Z$, the reciprocity law yields an alternative method for computing $S(a,b)$ by following the Euclidean algorithm.

Given coprime integers $a$ and $b$ with $b>0$, it is natural to ask what the denominator of $S(a,b)$ is. It follows from expansion of the original definition and some rearrangement (see, for example, \cite[p. 27]{Rademacher72}) that $bS(a,b) \in \Z$, and so the denominator of $S(a,b)$ when written as a reduced fraction is a divisor of $b$. In fact, a more exact statement holds:
\begin{thm}[Girstmair, \cite{Girstmair17}]
\label{denom}
Suppose $a$ and $b$ are coprime integers with $b>0$, and suppose we can write $S(a,b)$ as
\[
S(a,b) = \frac{k}{q}
\]
for $k,q \in \Z$, $q>0$, and $\gcd(k,q)=1$. Then
\[
q = \frac{b}{\gcd(b,a^2+1)}.
\]
In particular, given positive integers $b,q$ and an integer $a$ with $\gcd(a,b)=1$, $S(a,b)$ has the form $\frac{k}{q}$ for some $k\in \Z$ with $\gcd(k,q)=1$ if and only if
\[
b = \frac{q(a^2+1)}{t}
\]
for some positive integer $t$ with $\gcd(t,q)=1$.
\end{thm}
Furthermore, Girstmair proved that the problem of classifying the set of $k\in \Z$ such that $\gcd(k,q)=1$ and $\frac{k}{q}$ is the value of a normalized Dedekind sum can be reduced to classification $\pmod{q(q^2-1)}$. More precisely, he proved the following.

\begin{thm}[Girstmair, \cite{Girstmair17}]
\label{finitemod}
Let $k$ and $q$ be coprime integers with $q\ge 2$. If $k' \in \Z$ and 
\[
k' \equiv k \pmod{q(q^2-1)},
\]
then $\frac{k}{q}$ is the value of a normalized Dedekind sum if and only if $\frac{k'}{q}$ is the value of a normalized Dedekind sum.
\end{thm}
Thus for each fixed value of $q$, the question of determining the set of normalized Dedekind sums with denominator $q$ is reduced to determining the finite set of residue classes $\pmod{q(q^2-1)}$ which represent the possible numerators. In fact, Girstmair provides necessary conditions on the numerators of such normalized Dedekind sums and conjectures that these are the only restrictions on the set of normalized Dedekind sums with a given denominator.
\begin{thm}[Girstmair, \cite{Girstmair17}]
\label{qgood}
Suppose for integers $a,b,k,q$ with $b\ge 1,q\ge 2,$ and $\gcd(a,b)=\gcd(k,q)=1$, it holds that $S(a,b) = \frac{k}{q}$. Then
\begin{itemize}
\item If $3\nmid q$, then $3\mid k$.
\item If $2\nmid q$, then
\[
k\equiv\begin{cases}
2\pmod{4} & q\equiv 3 \pmod{4}\\
0\pmod{8} & q \text{ is a square}\\
0 \pmod{4} & \text{else.}
\end{cases}
\]
\end{itemize}
\end{thm}

\begin{conj}[Girstmair, \cite{Girstmair17}]
\label{conj1}
For integers $k,q$ with $q\ge 2$ and $\gcd(k,q)=1$, there exist integers $a,b$ with $b\ge 1$ and $\gcd(a,b)=1$ such that
\[
S(a,b) = \frac{k}{q}
\]
if and only if the conditions of Theorem \ref{qgood} hold.
\end{conj}
This conjecture is our main focus of study.

\section{Reduction to \texorpdfstring{$\pmod{q^2-1}$}{mod}}
\label{sec:reduce}
We first reduce the problem from a question of existence $\pmod{q(q^2-1)}$ to an existence problem $\pmod{(q^2-1)}$. In proving Theorem \ref{finitemod}, Girstmair uses the generalized three-term relation (see \cite{Girstmair98}) to obtain a new formula for the normalized Dedekind sum in light of Theorem \ref{denom}.

\begin{lemma}[Girstmair, \cite{Girstmair17}]
\label{girstmairidentity}
Let $q,t$ be positive integers such that $q\ge 2$ and $\gcd(t,q)=1$, let $a$ be an integer such that $\gcd(a,q)=1$ and $t\mid a^2+1$, and let $b=\frac{q(a^2+1)}{t}$. Then
\[
S(a,b) = \frac{(q^2-1)a}{tq}-S(aq,t)+S(at^{*},q),
\]
where $t^{*}$ is any integer such that $t t^{*}\equiv 1 \pmod{q}$.
\end{lemma}
Now fix an integer $a$, no longer necessarily coprime with $q$, such that $t\mid a^2+1$. We consider Dedekind sums in the form $S(a',b')$ where $a'\equiv a \pmod{t}$ and $\gcd(a', q)=1$, as motivated by \cite{Girstmair17}. If $a'= a+tj$ for some $j$ such that $\gcd(a',q)=1$ and $b' = \frac{q(a'^2+1)}{t}$, then $t\mid a'^2+1$ still holds, so $S(a',b')$ has reduced denominator $q$. Its numerator is
\begin{align*}
q S(a',b') &=\frac{(q^2-1)a}{t} +(q^2-1)j - qS(a'q,t)+qS(a't^{*},q)\\
&=\frac{(q^2-1)a}{t} +(q^2-1)j - qS(aq,t)+qS(at^{*}+j,q)\\
\begin{split}&= \left(\frac{(q^2-1)a}{t}-(q^2-1)at^{*}-qS(aq,t)\right)\\
&\quad +\left((q^2-1)(at^{*}+j)+q S(at^{*}+j,q)\right).
\end{split}
\end{align*}
The first expression is only dependent on $a,t,$ and $t^{*}$, while the second is only dependent on $at^{*}+j$. This motivates analyzing the two expressions separately. By considering a fixed $a$ and all $a'$ in the form $a'=a+tj$ for varying $j$, we can isolate the behavior of $qS(a',b')$ dependent on $a' \pmod{t}$ in the first term and the behavior dependent on $a'\pmod{q}$ in the second term.
\begin{defn}
Given a fixed integer $q\ge 2$, we define the functions $\lambda(a,t,t^{*})$ and $\Delta(\ell)$ by
\[
\lambda(a,t,t^{*}) = \frac{(q^2-1)a}{t} - (q^2-1)at^{*} - qS(aq,t)
\]
and
\[
\Delta(\ell) = (q^2-1)\ell+qS(\ell,q)
\]
where $\lambda$ is defined on all triples of integers $a,t,t^{*}$ such that $t>0$, $t\mid a^2+1$, and $tt^{*}\equiv 1\pmod{q}$, and $\Delta$ is defined on all integers $\ell$ such that $\gcd(\ell,q)=1$.
\end{defn}
Note that $\gcd(a+tj,q)=1$ if and only if $\gcd(at^{*}+j,q)=1$, so $\Delta(at^{*}+j)$ is well defined if and only if $\gcd(a+tj,q)=1$. We can now give our central decomposition of the numerator of $S(a',b').$
\begin{prop}
\label{split}
If integers $a,q,t,j,a',$ and $b'$ are given such that $t$ is positive, $q\ge 2$, $\gcd(t,q)=1$, $a'=a+tj$, $b' = \frac{q(a'^2+1)}{t}$, and $\gcd(a',q)=1$, then $S(a',b')$ has reduced denominator $q$ and numerator
\[
qS(a',b') = \lambda(a,t,t^{*}) +\Delta(at^{*}+j).
\]
Furthermore, $\lambda(a,t,t^{*})\in \Z$ and $\Delta(at^{*}+j)\in \Z$.
\end{prop}
\begin{proof}
We've shown that the identity for $qS(a',b')$ holds and that the denominator of $S(a',b')$ is reduced, so it suffices to show that $\lambda(a,t,t^{*})\in \Z$ and $\Delta(at^{*}+j) \in \Z$. But the denominator of $S(\ell,q)$ divides $q$, which implies that $qS(\ell,q)$ and therefore $\Delta(at^{*}+j)$ are integers. Furthermore $qS(a',b')\in \Z$, which implies $\lambda(a,t,t^{*}) = qS(a',b')-\Delta(at^{*}+j) \in \Z$.
\end{proof}
To understand the possible values of $qS(a',b')\pmod{q(q^2-1)}$, we analyze the behavior of $\lambda(a,t,t^{*})$ and $\Delta(\ell)$ when reduced $\pmod{q}$ and $\pmod{q^2-1}$. 
\begin{lemma}
\label{deltamodq}
If $\ell_1,\ell_2\in \Z$ such that $\gcd(\ell_1,q)=\gcd(\ell_2,q)=1$ and $\ell_1\equiv\ell_2\pmod{q}$, then $\Delta(\ell_1)\equiv \Delta(\ell_2)\pmod{q}$. Furthermore, as $\ell\pmod{q}$ ranges over $(\Z/q\Z)^{\times}$, $\Delta(\ell)\pmod{q}$ ranges over $(\Z/q\Z)^{\times}$. In other words, $\Delta$ is a bijection when considered as a mapping from $(\Z/q\Z)^{\times}$ to itself.

\end{lemma}
\begin{proof}
The first assertion is clear from the periodicity of $S(\ell,q)$ in $\ell$. Now we claim
\[
q S(\ell,q) \equiv \ell+\ell^{*} \pmod{q}
\]
where $\ell^{*}$ is an integer such that $\ell \ell^{*} \equiv 1\pmod{q}$. Indeed, if $c$ is an integer such that $\ell\ell^{*}-cq=1$, and we define
\[
M = \begin{pmatrix}
\ell^{*} & c\\
q & \ell
\end{pmatrix} \in \SL_2(\Z),
\]
then (as $q>0$) the Rademacher function $\Phi:\SL_2(\Z)\to \Z$ (see, for example, \cite[p. 50]{Rademacher72}) is given by
\[
\Phi(M) = \frac{\ell+\ell^{*}}{q} - S(\ell,q).
\]
But since $\Phi(M) \in \Z$, we have $qS(\ell,q) \equiv \ell+\ell^{*} \pmod{q}$.

Now reducing $\Delta(\ell)\pmod{q}$, we get
\begin{align*}
\Delta(\ell) &\equiv (q^2-1)\ell+qS(\ell,q) \\
& \equiv -\ell + (\ell+\ell^{*})\\
&\equiv \ell^{*} \pmod{q}
\end{align*}
implying $\Delta(\ell)$ takes the value of exactly the invertible residues $\pmod{q}$ as $\ell$ ranges over the invertible residues $\pmod{q}$.

\end{proof}
Now we consider $\lambda(a,t,t^{*})\pmod{q}$.
\begin{lemma}
\label{lambdamodq}
For any integers $a,t,t^{*}$ with $t>0$, $t\mid a^2+1$, $\gcd(aq,t)=1$, and $tt^{*}\equiv 1 \pmod{q}$, we have
\[
\lambda(a,t,t^{*}) \equiv 0 \pmod{q}.
\]
\end{lemma}
\begin{proof}
Note that $\gcd(q,t)=1$. Then
\begin{align*}
t \lambda(a,t,t^{*})&\equiv (q^2-1)a(1-tt^{*}) - qtS(aq,t)\\
&\equiv 0\pmod{q}
\end{align*}
because $tt^{*}\equiv 1\pmod{q}$ and $tS(aq,t)\in \Z$.
\end{proof}
Finally, we consider $\lambda(a,t,t^{*})\pmod{q^2-1}$.
 
\begin{defn}
Define $\Lambda_q \subseteq \Z/(q^2-1)\Z$ by
\[
\Lambda_q = \left\{\frac{(q^2-1)a}{t} - qS(aq,t)\pmod{q^2-1}:\gcd(aq,t)=1, t \mid a^2+1,t>0\right\}.
\]
\end{defn}
\begin{rem}
Alternatively, $\Lambda_q$ gives the range of $\lambda$:
\[
\Lambda_q = \left\{\lambda(a,t,t^{*})\pmod{q^2-1} : \gcd(aq,t)=1,t\mid a^2+1,tt^{*}\equiv 1 \pmod{q},t>0\right\}.
\]
\end{rem}
The characterizations are equivalent because the $(q^2-1)at^{*}$ term vanishes $\pmod{q^2-1}$.

\begin{defn}
Define $\Gamma_q\subseteq \Z/(q^2-1)\Z$ by 
\[
\Gamma_q = \begin{cases}
\{24s: s \in \Z/(q^2-1)\Z\}& q\text{ is a square}\\
\{12s: s \in \Z/(q^2-1)\Z\}&\text{else.}
\end{cases}
\]
\end{defn}
\begin{lemma}
For any integer $q\ge 2$,
\[
\Lambda_q\subseteq\Gamma_q
\]
\end{lemma}
\begin{proof}
It suffices to show the following claims:
\begin{itemize}
\item If $3\nmid q$, then $3\mid \lambda(a,t,t^{*})$.
\item If $2\nmid q$, then
\[
\lambda(a,t,t^{*})\equiv \begin{cases}
0 \pmod{8}& q\text{ is a square}\\
0 \pmod{4}&\text{else.}
\end{cases}
\]
To prove both, we note that for any $a$ such that $t\mid a^2+1$, there exists some $a' = a+tj$ such that $\gcd(at^{*}+j,q)=1$ and therefore $\gcd(a',q)=1$, recalling that $\gcd(t,q)=1$. Then by Proposition \ref{split}, we may write
\[
\lambda(a,t,t^{*}) = qS(a',b') - \Delta(at^{*}+j)
\]
for $b'=\displaystyle\frac{(a'^2+1)q}{t}$.

First, suppose $3\nmid q$. By Theorem $\ref{qgood}$, $3\mid qS(a',b')$, and similarly $3 \mid q S(at^{*}+j,q)$, as the denominator of $S(at^{*}+j,q)$ is a divisor of $q$, which is coprime with $3$. But $3\mid (q^2-1)$, so $3 \mid \Delta(at^{*}+j)$ and $3 \mid \lambda(a,t,t^{*})$.

Now suppose $2\nmid q$. By Theorem \ref{qgood},
\[
qS(a',b') \equiv q-1 \pmod{4}.
\]
Furthermore, $8 \mid q^2-1$, and by \cite[p. 34]{Rademacher72},
\[
qS(at^{*}+j,q)\equiv q+1-2\left(\frac{at^{*}+j}{q}\right)\pmod{8},
\]
where $\left(\frac{at^{*}+j}{q}\right)$ denotes the Jacobi symbol. This implies
\[
\Delta(at^{*}+j)\equiv qS(at^{*}+j,q)\equiv q+1-2\left(\frac{at^{*}+j}{q}\right)\pmod{8}
\]
and
\begin{align*}
\lambda(a,t,t^{*})&\equiv qS(a',b')-\Delta(at^{*}+j) \\
&\equiv \left(q-1\right)-\left(q+1-2\left(\frac{at^{*}+j}{q}\right)\right)\\
&\equiv -2+2\left(\frac{at^{*}+j}{q}\right)\\
&\equiv 0\pmod{4}.
\end{align*}
\end{itemize}
Furthermore, if $q$ is a square, then by Theorem \ref{qgood}, $qS(a',b')\equiv 0 \pmod{8}$, and since the Jacobi symbol is multiplicative in the denominator, $\left(\frac{at^{*}+j}{q}\right) \equiv 1 \pmod{8}$. Thus
\begin{align*}
qS(at^{*}+j,q) &\equiv q+1-2\left(\frac{at^{*}+j}{q}\right) \equiv 1+1-2\left( 1 \right) \equiv 0\pmod{8}
\end{align*}
implying $\lambda(a,t,t^{*})\equiv qS(a',b') - \Delta(at^{*}+j)\equiv 0\pmod{8}$, as desired.

\end{proof}

Finally, we present our main conjecture, which specifies the range of $\lambda(a,t,t^{*})$ and, as we will see, implies Conjecture \ref{conj1}.
\begin{conj}
\label{conj2}
For any integer $q\ge 2$,
\[
\Lambda_q=\Gamma_q
\]
\end{conj}
In a sense, Conjecture \ref{conj2} states that $\lambda(a,t,t^{*})$ takes on all possible values $\pmod{q^2-1}$ after accounting for the $\pmod{3}$ and $\pmod{8}$ restrictions of $\Gamma_q$. Similarly, this implies that the normalized Dedekind sum takes on all possible values after accounting for the local restrictions given by Theorem \ref{qgood}.
\begin{thm}
For each $q$, Conjecture \ref{conj2} implies Conjecture \ref{conj1}.
\end{thm}
\begin{proof}
Suppose $a,q,t,a',b',$ and $j$ are as in Proposition \ref{split}. We have established the following facts:
\begin{itemize}
\item The denominator of $S(a',b')$ is $q$, so its numerator is
\[
qS(a',b') = \lambda(a,t,t^{*}) + \Delta(at^{*}+j).
\]
\item
It always holds that
\[
\lambda(a,t,t^{*})\equiv 0 \pmod{q}.
\]
\item
As $at^{*}+j$ ranges over all invertible residues $\pmod{q}$, the remainder of $\Delta(at^{*}+j)$ also ranges over all invertible residues $\pmod{q}$.
\item
Assuming Conjecture \ref{conj2}, the range of $\lambda(a,t,t^{*})$ when reduced $\pmod{q^2-1}$ is $\Lambda_q=\Gamma_q$.
\end{itemize}
For the rest of the proof, assume Conjecture \ref{conj2} does hold. It is at this point that we finally use our full flexibility to generate pairs $a',b'$ based on a choice of $a$ and a variable choice of $j$. Given any $a,t,$ and $t^{*}$, we may choose $j$ independently to yield new values $a'$ and $b'$ as long as $\gcd(a',q)=\gcd(at^{*}+j,q)=1$. Then for any residue $\lambda(a,t,t^{*})\pmod{q(q^2-1)}$ and any residue $\Delta(\ell)\pmod{q(q^2-1)}$, the sum $\lambda(a,t,t^{*}) +\Delta(\ell)\pmod{q(q^2-1)}$ is the residue $\pmod{q(q^2-1)}$ of the numerator $qS(a',b')$ of some Dedekind sum $S(a',b')$ with reduced denominator $q$.

More precisely, since $\Lambda_q=\Gamma_q$ but $\lambda(a,t,t^{*})\equiv 0 \pmod{q}$ always, we have that $\lambda(a,t,t^{*})\pmod{q(q^2-1)}$ can take any value from $\widetilde{\Gamma}_q \subseteq \Z/q(q^2-1)\Z$, defined by
\[
\widetilde{\Gamma}_q = \left\{s \in \Z/q(q^2-1)\Z: s \equiv 0 \pmod{q}, s \pmod{q^2-1} \in \Gamma_q\right\},
\]
where each element in $\Gamma_q$ lifts uniquely to an element of $\widetilde{\Gamma}_q$ by the Chinese Remainder Theorem. Again because $\Delta(\ell)$ is periodic in $\ell$ with period $q$, it can be considered as a map of $\ell \in (\Z/q\Z)^{\times}$. Then if we define $\phi: \widetilde{\Gamma}_q\times (\Z/q\Z)^{\times}\to \Z/(q(q^2-1))\Z$ by
\[
\phi(\lambda,\ell) = \lambda+\Delta(\ell)\pmod{q(q^2-1)},
\]
we must have that the set of possible numerators $qS(a',b')$ is the image of $\phi$.

We claim $\phi$ is injective. Indeed, suppose there exist $\lambda_1,\lambda_2 \in \widetilde{\Gamma}_q$ and $\ell_1,\ell_2 \in (\Z/q\Z)^{\times}$ such that 
\[
\lambda_1+\Delta(\ell_1) \equiv \lambda_2 + \Delta(\ell_2)\pmod{q(q^2-1)}.
\]
Since $\lambda_1\equiv \lambda_2 \equiv 0 \pmod{q}$, reducing $\pmod{q}$ yields
\[
\Delta(\ell_1)\equiv \Delta(\ell_2) \pmod{q}
\]
so $\ell_1\equiv \ell_2 \pmod{q}$. This implies $\Delta(\ell_1) = \Delta(\ell_2)$, and so
\[
\lambda_1\equiv \lambda_2 \pmod{q(q^2-1)}.
\]
Then $\phi$ is injective, so the number of residue classes of possible numerators $\pmod{q(q^2-1)}$ is $|\widetilde{\Gamma}_q|\varphi(q) = |\Gamma_q|\varphi(q)$, where $\varphi$ denotes Euler's totient function. We claim that the number of residue classes $k$ satisfying the conditions of Theorem \ref{qgood} is also $|\Gamma_q|\varphi(q)$. Indeed, note that $2\nmid q$ if and only if $8 \mid q^2-1$ and $3\nmid  q$ if and only if $3\mid q^2-1$. Since $\gcd(q,q^2-1)=1$, by the Chinese Remainder Theorem the number of such residue classes $k\pmod{q(q^2-1)}$ is
\[
\varphi(q)\left(\frac{q^2-1}{c_2(q)c_3(q)}\right)
\]
where
\[
c_2(q) = \begin{cases}
1 & 2\mid q\\
4 & 2\nmid q \text{ and }q\text{ nonsquare} \\
8 & 2\nmid q \text{ and }q\text{ square}
\end{cases}
\]
and
\[
c_3(q) = \begin{cases}
1 & 3 \mid q\\
3 & 3 \nmid q.
\end{cases}
\]
But $8\mid q^2-1$ if $2\nmid q$ and $3\mid q^2-1$ if $3\nmid q$, while $2\nmid q^2-1$ if $2\mid q$ and $3\nmid q^2-1$ if $3\mid q$, so 
\[
|\Gamma_q| = \frac{q^2-1}{c_2(q)c_3(q)}
\]
as well. Thus all possible residue classes $k \pmod{q(q^2-1)}$ satisfying the conditions of Theorem $\ref{qgood}$ are achievable as numerators of normalized Dedekind sums of denominator $q$. This, together with Theorem \ref{finitemod}, establishes Conjecture \ref{conj1}. 

\end{proof}
\section{A Partial Resolution of Conjecture \ref{conj2}}
\label{sec:partial}

In this section, we introduce identities that will allow us to prove Conjecture \ref{conj2} (and therefore Conjecture \ref{conj1}) in specific cases, such as $q$ even or $q$ an odd square divisible by $3$ or $5$. We also hypothesize an approach to proving Conjecture \ref{conj2} in general and provide proof by computer verification for all $2\le q \le 200$.

\begin{defn}
For a given positive integer $q$ and an integer $a$ such that $\gcd(q,a^2+1)=1$ let
\[
f(a) = \frac{(q^2-1) a}{a^2+1}-q S(aq,a^2+1).
\]
\end{defn}
Note that by setting $t=a^2+1$, we have $f(a)\pmod{q^2-1} \in \Lambda_q$ for any integer $a$ such that $\gcd(q,a^2+1)=1$ (recall $f(a)$ is an integer). 

The following identity reveals that $f(a)$ is a piecewise linear function depending on the residue class of $a\pmod{q}$.
\begin{lemma}
\label{linear}
If an integer $a$ satisfies $\gcd(q,a^2+1)=1$, then 
\[
f(a) = (g^2-1)a+S(a_1,q_1)+S(-a_1g^2-a_1^{*},q_1)
\]
where $g = \gcd(a,q)>0,a=ga_1,q=gq_1$, and $a_1a_1^{*}\equiv 1 \pmod{q_1}$.
\end{lemma}
\begin{proof}
Note that $f(a)$ is an odd function, so without loss of generality we may assume that $a>0$. (The case $a=0$ is trivial.) By the reciprocity law, we have
\[
S(aq,a^2+1)+S(a^2+1,aq) = \frac{aq}{a^2+1}+\frac{a^2+1}{aq}+\frac{1}{aq(a^2+1)}-3.
\]

Next, we apply the three-term relation (see \cite{Girstmair98}) to $S(a^2+1,aq)$ and $S(a_1,q_1)$. Suppose integers $j$ and $k$ satisfy
\[
-a_1j+q_1k=1,
\]
and let
\[
r = -aqk+(a^2+1)j.
\]
Then 
\[
r = -ag(q_1k-a_1j)+j = -ag+j
\]
and in particular
\[
r\equiv -a_1g^2-a_1^{*}\pmod{q_1}.
\]
But 
\[
(a^2+1)q_1 - (aq)a_1=q_1
\]
so
\[
S(a^2+1,aq) = S(a_1,q_1)+S(-a_1g^2-a_1^{*},q_1) + \frac{(aq)^2+2q_1^2}{aqq_1^2}-3.
\]
Combining this with the reciprocity law directly implies the desired result.
\end{proof}
\begin{cor}
If $a$ is an integer such that $\gcd(a^2+1,q)=1$ and $g = \gcd(a,q)$, then 
\[
f(a+mq) = f(a)+mq(g^2-1)
\]
for any integer $m$.
\end{cor}
\begin{proof}
Note that $\gcd(a^2+1,q)=1$ implies $\gcd((a+mq)^2+1,q)=1$, and furthermore $\gcd(a,q) = \gcd(a+mq,q)$. Thus the conclusion of Lemma \ref{linear} holds for $a+mq$ with the same $g$ as $a$. But $S(a_1,q_1)+S(-a_1g^2-a_1^{*},q_1)$ only depends on the residue class of $a\pmod{q}$, which implies the desired result.
\end{proof}

As a consequence, we have the following.

\begin{thm}
\label{mainresult}
Conjecture \ref{conj2}, and thus Conjecture \ref{conj1}, holds for $q$ even or a square divisible by $3$ or $5$.
\end{thm}
\begin{proof}
First, suppose $q$ is even. If there exists an integer $a$ such that $\gcd(a,q) = 2$ and $\gcd(a^2+1,q)=1$, then $f(a+mq)=f(a)+3mq$ for any integer $m$. But $f(a+mq)\pmod{q^2-1} \in \Lambda_q$ for any integer $m$, and $q$ is invertible $\pmod{q^2-1}$, which implies $f(a)+3m \pmod{q^2-1} \in \Lambda_q$ for all integers $m$. Now for $q$ even, $\Gamma_q$ consists of multiples of $3\pmod{q^2-1}$, so in particular $|\Lambda_q|\ge |\Gamma_q|$. This, along with our prior assumption that $\Lambda_q\subseteq \Gamma_q$, would imply $\Lambda_q=\Gamma_q$. 

So it suffices in this case to show that there exists an integer $a$ such that $\gcd(a,q)=2$ and $\gcd(a^2+1,q)=1$. For each prime $p_i>2$ dividing $q$, we may choose a residue $r_i\pmod{p_i}$ such that $r_i\not\equiv 0 \pmod{p_i}$ and $r_i^2+1\not\equiv 0 \pmod{p_i}$. This is because the two congruences only eliminate three of the $p_i$ residues for $p_i>3$, while if $p_i=3$, we may choose $r_i=1$. Then by the Chinese Remainder Theorem we may choose any $a$ such that
\[
a \equiv 2 \pmod{4}
\]
and
\[
a \equiv r_i\pmod{p_i}
\]
for each $i$, which suffices.

Now suppose $q$ is an odd square such that $3\mid q$. In the same way, it suffices to find $a$ such that $\gcd(a^2+1,q)=1$ and $\gcd(a,q) = 3$, since $3^2-1=8$. Again for each $p_i>3$ dividing $q$, we may choose a remainder $r_i$ such that $r_i\not\equiv 0\pmod{p_i}$ and $r_i^2+1\not\equiv 0\pmod{p_i}$. By the Chinese Remainder Theorem, we can find our desired $a$ by taking $a \equiv 3 \pmod{9}$ and $a\equiv r_i\pmod{p_i}$.

Finally, suppose $q$ is an odd square such that $5\mid q$ and $3\nmid q$. The argument is nearly identical to the case $3\mid q$: it suffices to find $a$ such that $\gcd(a^2+1,q)=1$ and $\gcd(a,q)=5$, since $5^2-1=24$. Taking $a\equiv 5 \pmod{25}$ and applying the Chinese Remainder Theorem in the same way, we recover the desired value of $a$.

\end{proof}

The specification $t=a^2+1$ is too strong to prove Conjecture \ref{conj2} in general. However, generalizing to $t=x^2+y^2$ yields enough flexibility to prove the conjecture for small values of $q$ (and possibly all $q$).
\begin{defn}
For a given integer $q\ge 2$ and integers $x$ and $y$ such that $\gcd(q,x^2+y^2)=\gcd(x,y)=1$, let
\[
h(x,y) = \frac{(q^2-1)xy^{*}}{x^2+y^2}-qS(qxy^{*},x^2+y^2)\pmod{q^2-1}
\]
where $yy^{*}\equiv 1 \pmod{x^2+y^2}$. (Note that $y^{*}\pmod{x^2+y^2}$ determines the expression $\pmod{q^2-1}$, so the choice of $y^{*}$ doesn't matter.)
\end{defn}
This is exactly $\lambda(a,t,t^{*})\pmod{q^2-1}$ for $t = x^2+y^2$ and $a = xy^{*}$. Note that this generalizes $f(a)$, as $h(x,1)=f(x)\pmod{q^2-1}$. In our case, this interpretation is useful for verifying small cases of Conjecture \ref{conj2} computationally.
\begin{thm}
\label{computer}
For all integers $q$ such that $2\le q\le 200$, if $x$ and $y$ range over pairs of coprime integers such that $\gcd(x^2+y^2,q)=1$, then $h(x,y)$ ranges over all of $\Gamma_q$.
\end{thm}
\begin{proof}
A naive computer search over pairs of coprime positive integers $x$ and $y$ suffices for the values of $q$ not already covered by Theorem \ref{mainresult}. Our Sage program terminates in under an hour with $x+y$ achieving the largest value when $(x,y) = (1576,1511)$ for $q=189$.
\end{proof}
\begin{cor}
Conjectures \ref{conj2} and \ref{conj1} hold for all positive integers $2 \le q\le 200$.
\end{cor}

For future work, we would hope to see a proof for Conjecture \ref{conj2}, and therefore Conjecture \ref{conj1}, in full generality. We conjecture that setting $t=x^2+y^2$ and $a=xy^{*}$ as in Theorem \ref{computer} is enough to do so:
\begin{conj}
Fix $q\ge 2$. If $x$ and $y$ vary over pairs of coprime integers such that $\gcd(q,x^2+y^2)=1$, then $h(x,y)$ ranges over $\Gamma_q$.
\end{conj}
Once again, as $h(x,y) \in \Lambda_q$, this would imply Conjecture \ref{conj2}.

In fact, we believe that $h(x,y)$ can achieve all values in $\Gamma_q$ with one fewer parameter. By setting $t = (x^2+1)(q^2+1) = (x+q)^2+(xq-1)^2$, we have
\begin{conj}
Fix $q\ge 2$. If $x$ ranges over all integers such that $\gcd(x+q,xq-1)=1$ and $\gcd(x^2+1,q)=1$, then $h(x+q,xq-1)$ ranges over $\Gamma_q$.
\end{conj}

More generally, we conjecture that for any positive integer $r$, we can set $t = (x^2+1)(r^2q^2+1) = (x+rq)^2+(xrq-1)^2$ and get the same result.
\begin{conj}
Fix $q\ge 2$ and $r\ge 1$. If $x$ ranges over all integers such that $\gcd(x+rq,xrq-1)=1$ and $\gcd(x^2+1,q)=1$, then $h(x+rq,xrq-1)$ ranges over $\Gamma_q$.
\end{conj}

\section{Acknowledgments}
This work was supported by NSF grant DMS-1659047 as part of the Duluth Research Experience for Undergraduates (REU). The author would like to thank Joe Gallian for supervising the program, suggesting the problem, and providing helpful commentary on drafts of this paper.

\bibliographystyle{plain}
\bibliography{anc/dedekind}

\end{document}